\documentclass[a4paper,9pt]{article}
\usepackage{amsfonts}
\usepackage{amssymb,amsthm,color}
\usepackage{amsmath}
\usepackage{comment}

\usepackage[normalem]{ulem}

\newcommand{\R}{\mathbb{R}}
\newcommand{\N}{\mathbb{N}}

\newtheorem{theorem}{Theorem}

\newtheorem{definition}{Definition}

\newcommand{\bp}{\begin{proof}}
\newcommand{\ep}{\end{proof}}

\begin{document}
\title{On the $L^p$ norm of the torsion function}

\author{{M. van den Berg} \\
School of Mathematics, University of Bristol\\
University Walk, Bristol BS8 1TW\\
United Kingdom\\
\texttt{mamvdb@bristol.ac.uk}\\
\\
{T. Kappeler}\\
Institut f\"ur Mathematik, Universit\"at Z\"urich\\
 Winterthurerstrasse 190,
CH-8057 Z\"urich, Switzerland\\
 \texttt{thomas.kappeler@math.uzh.ch}}
\date{15 February 2018}\maketitle
\vskip 3truecm \indent
\begin{abstract}\noindent
Bounds are obtained for the $L^p$ norm of the torsion function
$v_{\Omega}$, i.e. the solution of $-\Delta v=1,\, v\in
H_0^1(\Omega),$ in terms of the Lebesgue measure of $\Omega$ and the
principal eigenvalue $\lambda_1(\Omega)$ of the Dirichlet Laplacian
acting in $L^2(\Omega)$. We show that these bounds are sharp for $1\le
p\le 2$.
\end{abstract}
\vskip 1truecm \noindent \ \ \ \ \ \ \ \  { Mathematics Subject
Classification (2000)}: 35J25, 35P99, 58J35.
\begin{center} \textbf{Keywords}: Torsion function, Dirichlet boundary condition.
\end{center}
\section{Introduction\label{sec1}}
Let $\Omega$ be a non-empty open set in Euclidean space $\R^m$ with boundary
$\partial\Omega$. It is
well-known (\cite {vdBC},\cite{vdB}) that if the bottom of the
Dirichlet Laplacian defined by
\begin{equation}\label{e1} \lambda_1(\Omega)=\inf_{\varphi\in
H_0^1(\Omega)\setminus\{0\}}\frac{\displaystyle\int_\Omega|D\varphi|^2}{\displaystyle\int_\Omega
\varphi^2}
\end{equation}
is bounded away from $0$, then
\begin{equation}\label{e2}
-\Delta v=1,\, v\in H_0^1(\Omega)
\end{equation}
has a unique solution denoted by $v_{\Omega}$. The function $v_{\Omega}$ is non-negative, pointwise increasing in $\Omega$, and satisfies,
\begin{equation}\label{e3}
\lambda_1(\Omega)^{-1}\le \|v_{\Omega}\|_{L^{\infty}(\Omega)}\le
(4+3m\log 2)\lambda_1(\Omega)^{-1}.
\end{equation}
The $m$-dependent constant in the right-hand side of \eqref{e3} has
subsequently been improved (\cite{GS},\cite{HV}).
We denote the optimal constant in the right-hand side of
\eqref{e3} by
\begin{equation}\label{e4}
\mathfrak{F}_{\infty}=\sup\{\lambda_1(\Omega)
\|v_{\Omega}\|_{L^{\infty}(\Omega)}:\Omega\,\textup{ open
in}\,\R^m,\,|\Omega|<\infty\},
\end{equation}
suppressing the $m$-dependence.
The torsional rigidity of $\Omega$ is defined by
\begin{equation*}
T_1(\Omega)=\int_{\Omega}v_{\Omega}.
\end{equation*}
It plays a key role in different parts of
analysis. For example the torsional rigidity of a cross section of a
beam appears in the computation of the angular change when a beam of
a given length and a given modulus of rigidity is exposed to a
twisting moment (\cite{Bandle},\cite{PSZ}). It also arises in the
definition of gamma convergence \cite{BB} and in the study of
minimal submanifolds \cite{MP}. Moreover, $T_1(\Omega)/|\Omega|$
equals $\mathbb{E}_x(\tau_{\Omega}),$ the expected lifetime $\tau_{\Omega}$ of Brownian motion in $\Omega$, when
averaged with respect to the uniform distribution over all starting
points $x\in\Omega$.

A classical inequality, e.g. \cite{PSZ}, asserts that the
function $F_1$ defined on the open sets in $\R^m$ with finite
Lebesgue measure
\begin{equation}\label{e6}
F_1(\Omega)=\frac{T_1(\Omega)\lambda_1(\Omega)}{\vert\Omega\vert}
\end{equation}
satisfies
\begin{equation}\label{e7}
F_1(\Omega)\le 1.
\end{equation}
Since $\Omega$ has finite Lebesgue measure $|\Omega|$, \eqref{e3} implies that $v\in L^p(\Omega)$ for $1\leq p \leq
\infty$. Moreover $\lambda_1(\Omega)$ is in that case the principal eigenvalue of the Dirichlet Laplacian.
Motivated by \eqref{e6} and \eqref{e7} we make the following
\begin{definition}\label{def1}\begin{itemize} \item[\textup{(i)}]
For $\Omega$ open in $\R^m$ with $0<|\Omega|<\infty$ and $1\le p<\infty$,
\begin{equation}\label{e8}
F_p(\Omega)=\frac{T_p(\Omega)\lambda_1(\Omega)}{|\Omega|^{1/p}},
\end{equation}
where
\begin{equation}\label{e9}
T_p(\Omega)=\Vert
v_{\Omega}\Vert_{L^p(\Omega)}=\bigg(\int_{\Omega}v_{\Omega}^p\bigg)^{1/p}.
\end{equation}
\item[\textup{(ii)}] For $\Omega$ open in $\R^m$ with $\lambda_1(\Omega)>0$,
\begin{equation*}
F_{\infty}(\Omega)=\Vert v_{\Omega}\Vert_{L^{\infty}(\Omega)}\lambda_1(\Omega).
\end{equation*}
\end{itemize}
\end{definition}
It follows from the Faber-Krahn inequality that if $|\Omega|<\infty$ then $\lambda_1(\Omega)>0$. The converse does not hold for if $\Omega$ is the union of infinitely many disjoint balls of radii $1$ then $\lambda_1(\Omega)>0$ but $\Omega$ has infinite measure.
Note that $2T_2^2(\Omega)/|\Omega|$ equals the second moment of the
expected lifetime of Brownian motion in $\Omega$, when
averaged with respect to the uniform distribution over all starting
points $x\in\Omega$.

Note that $\Omega\mapsto T_p(\Omega)$ is increasing while $\Omega\mapsto \lambda_1(\Omega)$ and $\Omega\mapsto \vert\Omega\vert^{-1/p}$ are decreasing. It is straightforward to verify that $F_p,\,1\le p\le \infty$ is invariant under homotheties. That is, if $\alpha>0,\, \alpha\Omega=\{x\in \R^m: x/\alpha\in\Omega\}$, then $F_p(\alpha\Omega)= F_p(\Omega)$.

Our main results are the following.
\begin{theorem}\label{the1}
Let $\Omega$ be an open set in $\R^m,\,m=1,2,3,...$ with $|\Omega|<\infty$.
\begin{itemize}
\item[\textup{(i)}] If $1\le p\le q\le \infty$ then,
\begin{equation}\label{e11}
F_p(\Omega)\le F_q(\Omega)\le \mathfrak{F}_{\infty}.
\end{equation}
\item[\textup{(ii)}] If $1\le p\le 2$ then,
\begin{equation}\label{e12}
F_p(\Omega)\le F_1(\Omega)^{1/p}\le 1.
\end{equation}
\end{itemize}
\end{theorem}

\noindent
\begin{definition}\label{def2}
For $1\le p\le \infty$,
\begin{itemize}
\item[\textup{(i)}]
\begin{equation}\label{e13}
\mathfrak{F}_p=\sup\{F_p(\Omega):\Omega\, \textup{open in}\, \R^m,
|\Omega|<\infty\},
\end{equation}
\item[\textup{(ii)}]
\begin{equation*}
\mathfrak{G}_p=\inf\{F_p(\Omega):\Omega\, \textup{open in}\, \R^m,
|\Omega|<\infty\},
\end{equation*}
\item[\textup{(iii)}]
\begin{equation*}
\mathfrak{F}_p^{\textup{convex}}=\sup\{F_p(\Omega):\Omega\, \textup{open, convex in}\, \R^m,
|\Omega|<\infty\},
\end{equation*}
\item[\textup{(iv)}]
\begin{equation*}
\mathfrak{G}_p^{\textup{convex}}=\inf\{F_p(\Omega):\Omega\, \textup{open, convex in}\, \R^m,
|\Omega|<\infty\}.
\end{equation*}
\end{itemize}
\end{definition}
It was shown in \cite{MvdB1} that $\mathfrak{G}_{\infty}=1$.
\begin{theorem}\label{the2}
If $m=1,2,3,...$, and if $1\le p < \infty,$ then
\begin{enumerate}
\item[\textup{(i)}]
\begin{equation*}
\mathfrak{G}_p=0.
\end{equation*}
\item[\textup{(ii)}]The mapping $p\mapsto \mathfrak{G}_p^{\textup{convex}}$ is non-decreasing, and
\begin{align}\label{e17}
\mathfrak{G}_p^{\textup{convex}}\ge
2^{-3}\pi^2m^{-(m+2p)/p}\bigg(\frac{\Gamma(\frac
m2+1)\Gamma(p+1)}{\Gamma(\frac m2+p+1)}\bigg)^{1/p}.
\end{align}
\end{enumerate}
\end{theorem}
It follows from \eqref{e17} that $\lim_{p\rightarrow\infty}\mathfrak{G}_p^{\textup{convex}}\ge\pi^2/8.$
This jibes with the result of \cite{P} that
\begin{equation*}
\mathfrak{G}_{\infty}^{\textup{convex}}=\frac{\pi^2}{8}.
\end{equation*}

A monotone increasing sequence of cuboids which exhausts the open connected set bounded by two parallel $(m-1)$- dimensional hyperplanes
is a minimising sequence for $\mathfrak{G}_{\infty}^{\textup{convex}}$. See also Theorem 2 in \cite{MvdB1}.

\begin{theorem}\label{the3}Let $m=2,3,...$.
\begin{itemize}
\item[\textup{(i)}]The mappings $p\mapsto \mathfrak{F}_p$, and $p\mapsto \mathfrak{F}_p^{\textup{convex}}$ are non-decreasing on $[1,\infty]$.
\item[\textup{(ii)}]If
\begin{equation*}
 p_m=\inf\{p\ge 1: \mathfrak{F}_p>1\},
\end{equation*}
and
\begin{equation*}
 p_m^{\textup{convex}}=\inf\{p\ge 1: \mathfrak{F}_p^{\textup{convex}}>1\},
\end{equation*}
then
\begin{equation}\label{e20}
2\le p_m\le p_m^{\textup{convex}}\le 8m.
\end{equation}
In particular $$\mathfrak{F}_p=1, \,\,1\le p\le p_m.$$
\item[\textup{(iii)}] Formula \eqref{e13} defining $\mathfrak{F}_p$ does not have a maximiser for $1\le p\le 2.$ The maximising sequence constructed in \textup{\cite{MvdB}} for $\mathfrak{F}_1$ is also a maximising sequence for
$\mathfrak{F}_p,\,1\le p\le p_m.$
\item[\textup{(iv)}]The mappings $p\mapsto \mathfrak{F}_p$, and $p\mapsto \mathfrak{F}_p^{\textup{convex}}$ are left-continuous on $(1,\infty]$.
\item[\textup{(v)}]If $n\in \N,\,1\le p,$ then
\begin{equation}\label{e21}
\mathfrak{F}_{p+n}\le \bigg(\frac{p+n}{4^np}\prod_{j=1}^n(p+j)\bigg)^{\frac{1}{p+n}}\mathfrak{F}_p^{\frac{p}{p+n}},
\end{equation}
\begin{equation}\label{e21a}
\mathfrak{F}_{p+n}^{\textup{convex}}\le \bigg(\frac{p+n}{4^np}\prod_{j=1}^n(p+j)\bigg)^{\frac{1}{p+n}}\bigg(\mathfrak{F}_p^{\textup{convex}}\bigg)^{\frac{p}{p+n}}.
\end{equation}
In particular if $1\le p\le 2$, then
\begin{equation}\label{e21b}
\mathfrak{F}_{p+1}\le \bigg(\frac{(p+1)^2}{4p}\bigg)^{\frac{1}{p+1}}.
\end{equation}
\item[\textup{(vi)}]
\begin{equation}\label{e21c}
\mathfrak{F}_{n}\le\bigg(\frac{n.n!}{4^{n-1}}\bigg)^{\frac{1}{n}},\, n\in \N,
\end{equation}
and $\mathfrak{F}_3\le 3^{2/3}/2=1.04004...$.
\item[\textup{(vii)}]$p\mapsto \mathfrak{F}_p$ is differentiable at $p=2$, with $\mathfrak{F}'_2=0.$
\item[\textup{(viii)}] If $1\le p\le 2$, then
\begin{equation}\label{e21d}
\mathfrak{F}_p^{\textup{convex}}\le \big(\mathfrak{F}_1^{\textup{convex}}\big)^{1/p}.
\end{equation}
\item[\textup{(ix)}]
For $m = 2,$
\begin{equation*}
p_2^{\textup{convex}}\ge 2.0186.
\end{equation*}
\end{itemize}
\end{theorem}

This paper is organised as follows. In Section \ref{sec2} we prove
Theorems \ref{the1} and \ref{the2}. The
proof of Theorem \ref{the3} will be given in Section \ref{sec3}.

We note that a general multiplicative inequality involving $T_p(\Omega), \lambda_1(\Omega)$ and $|\Omega|$ will involve three exponents. However, the requirement that it be invariant under homotheties reduces the number of exponents to two. In Section \ref{sec4} we briefly discuss this two-parameter family of inequalities, and determine which parameter pair yields a finite supremum.

\section{Proofs of Theorems \ref{the1},\ref{the2}\label{sec2}}
\noindent{\it Proof of Theorem \ref{the1}.}

\noindent(i) To prove \eqref{e11} for $1\le p\le q<\infty$ we use H\"older's inequality to obtain that
\begin{equation*}
\int_{\Omega}v_{\Omega}^p\le\bigg(\int_{\Omega}v_{\Omega}^q\bigg)^{p/q}|\Omega|^{(q-p)/q}.
\end{equation*}
So we have that
\begin{equation*}
\Vert v_{\Omega}\Vert_{L^p(\Omega)}\le\Vert v_{\Omega}\Vert_{L^q(\Omega)}|\Omega|^{\frac
1p-\frac 1q}.
\end{equation*}
This, together with \eqref{e8}, implies \eqref{e11}. In case $q=\infty$,
\begin{equation*}
\Vert v_{\Omega}\Vert_{L^p(\Omega)}\le \Vert v_{\Omega}\Vert_{L^{\infty}(\Omega)}|\Omega|^{1/p}.
\end{equation*}

\noindent(ii) To prove \eqref{e12} we observe that since $\Omega$ has finite
Lebesgue measure the spectrum of the Dirichlet Laplacian acting in
$L^2(\Omega)$ is discrete, and consists of an increasing sequence of
eigenvalues
\begin{equation*}
\{\lambda_1(\Omega)\le\lambda_2(\Omega)\le
\lambda_3(\Omega)\le....\},
\end{equation*}
accumulating at infinity, where we have included multiplicities. We denote a corresponding
orthonormal basis of eigenfunctions by
$\{\varphi_{j,\Omega},j=1,2,3,...\}$. The resolvent of the Dirichlet
Laplacian acting in $L^2(\Omega)$ is compact, and its kernel
$H_{\Omega}$ has an $L^2$-eigenfunction expansion given by
\begin{equation*}
H_{\Omega}(x,y)=\sum_{j=1}^{\infty}\frac{1}{\lambda_j(\Omega)}\varphi_{j,\Omega}(x)\varphi_{j,\Omega}(y).
\end{equation*}
So $v_{\Omega}$, defined by \eqref{e2}, is given by
\begin{equation*}
v_{\Omega}(x)=\sum_{j=1}^{\infty}\frac{1}{\lambda_j(\Omega)}\bigg(\int_{\Omega}\varphi_{j,\Omega}\bigg)\varphi_{j,\Omega}(x).
\end{equation*}
Since $v_{\Omega}\in L^2(\Omega)$ we have by orthonormality that
\begin{align}\label{e29}
\int_{\Omega}v_{\Omega}^2&=\int_{\Omega}dx\,\sum_{j=1}^{\infty}\sum_{k=1}^{\infty}\frac{1}{\lambda_j(\Omega)}\bigg(\int_{\Omega}\varphi_{j,\Omega}\bigg)
\varphi_{j,\Omega}(x)\frac{1}{\lambda_k(\Omega)}
\bigg(\int_{\Omega}\varphi_{k,\Omega}\bigg)\varphi_{k,\Omega}(x)\nonumber
\\
&=\sum_{j=1}^{\infty}\frac{1}{\lambda_j^2(\Omega)}\bigg(\int_{\Omega}\varphi_{j,\Omega}\bigg)^2\nonumber
\\
&\le\frac{1}{\lambda_1(\Omega)}\sum_{j=1}^{\infty}\frac{1}{\lambda_j(\Omega)}\bigg(\int_{\Omega}\varphi_{j,\Omega}\bigg)^2\nonumber \\ &=\frac{T_1(\Omega)}{\lambda_1(\Omega)}.
\end{align}
We conclude that
\begin{equation}\label{e30}
T_2(\Omega)\le \bigg(\frac{T_1(\Omega)}{\lambda_1(\Omega)}\bigg)^{1/2}.
\end{equation}
Multiplying both sides of the inequality above with $\lambda_1(\Omega)/|\Omega|^{1/2}$ we obtain that $F_2(\Omega)\le \big(F_1(\Omega)\big)^{1/2}$. By (i) $\big(F_1(\Omega)\big)^{1/2}\le\big(F_2(\Omega)\big)^{1/2}.$ This, together with the previous inequality, implies that $F_2(\Omega)\le 1.$
We now use H\"older's inequality, and interpolate with $0<\alpha<1, \rho>1$ as follows.
\begin{equation*}
\int_{\Omega}v_{\Omega}^p=\bigg(\int_{\Omega}v_{\Omega}^{\alpha p\rho}\bigg)^{1/\rho}\bigg(\int_{\Omega}v_{\Omega}^{(1-\alpha) p\rho/(\rho-1)}\bigg)^{(\rho-1)/\rho}.
\end{equation*}
Choosing $\alpha p\rho=2,(1-\alpha)p\rho/(\rho-1)=1$ gives that $\rho=(p-1)^{-1}$. Hence by \eqref{e30},
\begin{equation*}
T_p^p(\Omega)=\int_{\Omega}v_{\Omega}^p\le \bigg(T_2^2(\Omega)\bigg)^{p-1}\bigg(T_1(\Omega)\bigg)^{2-p}\le \frac{T_1(\Omega)}{\lambda_1(\Omega)^{p-1}}.
\end{equation*}
Multiplying both sides of the inequality above with $\lambda_1(\Omega)^p/|\Omega|$ gives that
\begin{equation*}
F_p^p(\Omega)\le F_1(\Omega).
\end{equation*}

 \hspace*{\fill }$\square $

\noindent{\it Proof of Theorem \ref{the2}.}

\noindent(i) We let $\Omega_n$ be the disjoint union of one ball of
radius $1$ and $n$ balls with radii $r_n$, with $r_n<1$. Then
\begin{equation*}
|\Omega_n|=\big(nr_n^m+1\big)|B_1|,
\end{equation*}
where $B_1=\{x\in \R^m:|x|<1\}$.
Since $r_n<1$ we have that
\begin{equation*}
\lambda_1(\Omega_n)=\lambda_1(B_{1}).
\end{equation*}
Since $T^p_p$ is additive on disjoint open sets we have by scaling that
\begin{equation*}
T^p_p(\Omega_n)=\big(nr_n^{2p+m}+1\big)T^p_p(B_1).
\end{equation*}
Therefore
\begin{align}\label{e34}
F^p_p(\Omega_n)&=\frac{\big(nr_n^{2p+m}+1\big)T^p_p(B_1)\lambda^p_1(B_1)}{\big(nr_n^m+1\big)|B_1|}\nonumber
\\ &=\frac{nr_n^{2p+m}+1}{nr_n^m+1}F^p_p(B_1)\nonumber \\ &\le
\big(r_n^{2p}+n^{-1}r_n^{-m}\big)F^p_p(B_1).
\end{align}
We now choose $r_n$ as to minimise the right-hand side of \eqref{e34},
\begin{equation*}
r_n=\bigg(\frac{m}{2pn}\bigg)^{1/(2p+m)}.
\end{equation*}
This gives that
\begin{equation*}
F^p_p(\Omega_n)\le \bigg(1+\frac{2p}{m}\bigg)\bigg(\frac{m}{2p}\bigg)^{2p/(2p+m)} n^{-2p/(2p+m))}F^p_p(B_1),
\end{equation*}
which implies the assertion.

\noindent(ii) The first part of the assertion follows directly by \eqref{e11}. To prove the second part we recall John's ellipsoid
theorem (\cite{J},\cite{RH}) which asserts the existence of an ellipsoid $\Upsilon$ with centre
$c$ such that $\Upsilon \subset \Omega \subset c + m(\Upsilon - c).$
Here $c + m(\Upsilon - c)=\{c + m(x - c) : x \in \Upsilon\}.$
This is the dilation of $\Upsilon$ by the factor $m$. $\Upsilon$ is the ellipsoid of maximal volume in $\Omega$. By
translating both $\Omega$ and $\Upsilon$ we may assume that
\begin{equation*}
\Upsilon=\{x\in \R^m: \sum_{i=1}^m \frac{x_i^2}{a_i^2}<1\}, \qquad
a_i>0,\quad i=1,\dots,m.
\end{equation*}
It is easily verified that the unique solution of \eqref{e2} for
$\Upsilon$ is given by
\begin{equation*}
v_{\Upsilon}(x)=2^{-1}\left(\sum_{i=1}^m\frac{1}{a_i^2}\right)^{-1}\left(1-\sum_{i=1}^m\frac{x_i^2}{a_i^2}\right).
\end{equation*}
By changing to spherical coordinates, we find that
\begin{equation*}
\int_{\Upsilon} v_{\Upsilon}^p=2^{-p}\omega_m\frac{\Gamma(\frac
m2+1)\Gamma(p+1)}{\Gamma(\frac
m2+p+1)}\left(\sum_{i=1}^m\frac{1}{a_i^2}\right)^{-p}\prod_{i=1}^ma_i,
\end{equation*}
where $\omega_m=|B_1|$.
Since $\Omega\mapsto v_{\Omega}$ is increasing we have by \eqref{e9} that $\Omega\mapsto T_p(\Omega)$ is increasing, and
\begin{align}\label{e40}
T_p(\Omega)&\ge T_p(\Upsilon)\nonumber \\ & =2^{-1}\omega_m^{1/p}\bigg(\frac{\Gamma(\frac
m2+1)\Gamma(p+1)}{\Gamma(\frac
m2+p+1)}\bigg)^{1/p}\left(\sum_{i=1}^m\frac{1}{a_i^2}\right)^{-1}\bigg(\prod_{i=1}^ma_i\bigg)^{1/p}.
\end{align}
Since $\Omega\subset m\Upsilon$,
\begin{equation}\label{e41}
|\Omega|\le\int_{m\Upsilon}dx=\omega_mm^m\prod_{i=1}^ma_i.
\end{equation}
By the monotonicity of Dirichlet eigenvalues, we have that
$\lambda_1(\Omega)\ge \lambda_1(m\Upsilon)$. The ellipsoid $m\Upsilon$
is contained  in a cuboid with lengths $2ma_1,\dots,2ma_m.$ So we
have that
\begin{equation}\label{e42}
\lambda_1(\Omega)\ge \frac{\pi^2}{4m^2}\sum_{i=1}^m\frac{1}{a_i^2}.
\end{equation}
Combining \eqref{e40}, \eqref{e41}, \eqref{e42}, and \eqref{e9}
gives \eqref{e17}. \hspace*{\fill }$\square $

\section{Proof of Theorem \ref{the3}\label{sec3}}

(i) It follows from the second inequality in \eqref{e11} that $\mathfrak{F}_q\le \mathfrak{F}_{\infty}$. Hence $F_p(\Omega)\le \mathfrak{F}_q\le \mathfrak{F}_{\infty}$. Taking subsequently the supremum over all $\Omega$ with finite measure we obtain the first assertion under
(i). As \eqref{e11} holds for all open sets with finite measure, it also holds for all bounded convex sets. Then, the preceding argument gives the second assertion under (i).

\noindent(ii) It follows from \eqref{e12} that $\mathfrak{F}_p\le 1, \, 1\le p\le 2$. In Theorem 1.2 of \cite{MvdB} it was shown that the bound $F_1(\Omega)\le 1$ is sharp. That is $\mathfrak{F}_1=1$.
This, together with (i), then implies that $\mathfrak{F}_p=1$ for $1\le p\le 2$. Hence $p_m\ge 2$. Since $\mathfrak{F}_p^{\textup{convex}}\le \mathfrak{F}_p$ we conclude the second inequality in \eqref{e20}. To prove the upper bound on $p_m^{\textup{convex}}$ we recall that
\begin{equation*}
v_{B_1}(x)=\frac{1-|x|^2}{2m}.
\end{equation*}
Hence, denoting by $(f)_{+}$ the positive part of a real-valued function $f$, we have that
\begin{align}\label{e44}
T_p(B_1)&=\bigg(\int_{[0,1]}dr\,m\omega_m\bigg(\frac{1-r^2}{2m}\bigg)^pr^{m-1}\bigg)^{1/p}\nonumber \\ &=\frac{(m\omega_m)^{1/p}}{2^{(p+1)/p}m}\bigg(\int_{[0,1]}d\theta (1-\theta)^p\theta^{(m-2)/2}\bigg)^{1/p}\nonumber \\ &
\ge \frac{(m\omega_m)^{1/p}}{2^{(p+1)/p}m}\bigg(\int_{[0,1]}d\theta (1-p\theta)_+\theta^{(m-2)/2}\bigg)^{1/p}\nonumber \\ &\ge\frac{2^{1/p}\omega_m^{1/p}}{2m(m+2)^{1/p}p^{m/(2p)}}\nonumber \\ &\ge
\frac{\omega_m^{1/p}}{2m^{(p+1)/p}p^{m/(2p)}}.
\end{align}
It follows that
\begin{equation*}
\mathfrak{F}_p\ge F_p(B_1)\ge \frac{j^2_{(m-2)/2}}{2m^{(p+1)/p}p^{m/(2p)}},
\end{equation*}
where $\lambda_1(B_1)=j^2_{(m-2)/2}$, and $j_{(m-2)/2}$ is the first positive zero of the Bessel function $J_{(m-2)/2}$.
Hence
\begin{align}\label{e46}
\mathfrak{F}_{8m}\ge \frac{j^2_{(m-2)/2}}{2m^{1+\frac{1}{8m}}(8m)^{\frac{1}{16}}}\ge \frac{j^2_{(m-2)/2}}{2^{\frac{19}{16}}m^{\frac98}}.
\end{align}
One verifies numerically that for $m=2,...,19,$ the right-hand side of \eqref{e46} is strictly greater than $1$.
Since $j^2_{(m-2)/2}\ge ((m-2)/2)^2$ (see inequality (1.6) in \cite{AE}) we have for $m\ge 20$ that $j^2_{(m-2)/2}> m^2/5$. But $m^{\frac78}\ge 5\cdot 2^{19/16}, m\ge 20$.

\noindent(iii) It was shown in \cite{MvdB} that the formula defining $\mathfrak{F}_1$ in \eqref{e13} does not have a maximiser.
Since by \eqref{e12}, $F_p(\Omega)\le F_1(\Omega)^{1/p}\le 1$ for any $1\le p \le 2$ and any open subset $\Omega \subset \R^m$ with $| \Omega | < \infty$,
none of the formulae defining $\mathfrak{F}_p,\, 1\le p\le 2$, have maximisers.
Clearly, the maximising sequence constructed in \cite{MvdB} for $\mathfrak{F}_1$ is a maximising sequence for $\mathfrak{F}_p,\, 1\le p\le p_m$.

\noindent(iv) To prove left-continuity we first fix $1<q<\infty$, and let $\epsilon>0$ be arbitrary. There exists an open set $\Omega_{q,\epsilon}\subset \R^m$ such that
\begin{equation}\label{e47}
\mathfrak{F}_q\ge F_q(\Omega_{q,\epsilon})\ge \mathfrak{F}_q-\frac{\epsilon}{2}.
\end{equation}
By scaling we may assume that $|\Omega_{q,\epsilon}|=1$. Let $p\in[1,q)$. Then
\begin{equation*}
\mathfrak{F}_q\ge\mathfrak{F}_p\ge F_p(\Omega_{q,\epsilon}),
\end{equation*}
and
\begin{align*}
\int_{\Omega_{q,\epsilon}}v^q_{\Omega_{q,\epsilon}}&\le \Vert v_{\Omega_{q,\epsilon}}\Vert^{q-p}_{L^{\infty}(\Omega_{q,\epsilon})}\int_{\Omega_{q,\epsilon}}v^p_{\Omega_{q,\epsilon}}\nonumber \\ &\le \mathfrak{F}_{\infty}^{q-p}\lambda_1^{p-q}(\Omega_{q,\epsilon})\int_{\Omega_{q,\epsilon}}v^p_{\Omega_{q,\epsilon}},
\end{align*}
implying that
\begin{equation*}
F_p(\Omega_{q,\epsilon})\ge \mathfrak{F}_{\infty}^{(p-q)/p}F_q(\Omega_{q,\epsilon})^{q/p}.
\end{equation*}
Since $p\mapsto F_p$ is increasing we have that
\begin{equation}\label{e51}
\lim_{p\uparrow q}F_p(\Omega_{q,\epsilon})\ge \mathfrak{F}_{\infty}^{(p-q)/p}F_q(\Omega_{q,\epsilon})^{q/p}.
\end{equation}
Since $p<q$, we have by the continuity of the right-hand side of \eqref{e51} in $p$, \eqref{e46} and \eqref{e47} that
\begin{equation*}
\mathfrak{F}_q\ge\lim_{p\uparrow q}\mathfrak{F}_p\ge F_q(\Omega_{q,\epsilon})\ge \mathfrak{F}_q-\frac{\epsilon}{2}.
\end{equation*}
Letting $\epsilon\downarrow 0$ concludes the proof for $1<q<\infty$.

To prove left-continuity at $q=\infty$ we let $\epsilon>0$ be arbitrary. By \eqref{e13} there exists an open set $\Omega_{\infty,\epsilon}$ such that
\begin{equation}\label{e53}
\mathfrak{F}_{\infty}\ge F_{\infty}(\Omega_{\infty,\epsilon})\ge \mathfrak{F}_{\infty}-\frac{\epsilon}{2}.
\end{equation}
Without loss of generality we may assume by scaling that $|\Omega_{\infty,\epsilon}|=1$. Then $v_{\Omega_{\infty,\epsilon}}\in L^p(\Omega_{\infty,\epsilon}),\, 1\le p\le \infty$, and
\begin{equation*}
F_{\infty}(\Omega_{\infty,\epsilon})=\lim_{p\rightarrow \infty} F_{p}(\Omega_{\infty,\epsilon}).
\end{equation*}
Hence there exists $p(\epsilon)<\infty$ such that
\begin{equation}\label{e55}
|F_{\infty}(\Omega_{\infty,\epsilon})- F_{p}(\Omega_{\infty,\epsilon})|\le \frac{\epsilon}{2},\quad p\ge p(\epsilon).
\end{equation}
This implies, by \eqref{e53} and \eqref{e55}, that
\begin{align}\label{e56}
\mathfrak{F}_p&\ge F_p(\Omega_{\infty,\epsilon}) \nonumber \\&
\ge F_{\infty}(\Omega_{\infty,\epsilon})-\frac{\epsilon}{2}\nonumber \\ &\ge \mathfrak{F}_{\infty}-\epsilon,\quad p\ge p(\epsilon).
\end{align}
Hence by (i) and \eqref{e56},
\begin{equation*}
\mathfrak{F}_{\infty}\ge \lim_{p\uparrow \infty}\mathfrak{F}_p\ge\mathfrak{F}_{\infty}-\epsilon.
\end{equation*}
The left-continuity at $\infty$ now follows since $\epsilon>0$ was arbitrary.

\noindent(v) Let $n\in \N,\, p\ge 1$. Without loss of generality we may assume that $|\Omega|=1$. An integration by parts shows that
\begin{equation}\label{e58}
\int_{\Omega}v_{\Omega}^p=-\int_{\Omega}v_{\Omega}^p\Delta v_{\Omega}=p\int_{\Omega}v_{\Omega}^{p-1} |Dv_{\Omega}|^2=\frac{4p}{(p+1)^2}\int_{\Omega}|Dv_{\Omega}^{(p+1)/2}|^2.
\end{equation}
By \eqref{e1}
\begin{equation}\label{e59} \lambda_1(\Omega)\le\frac{\displaystyle\int_\Omega|Dv_{\Omega}^{(p+1)/2}|^2}{\displaystyle\int_\Omega
v_{\Omega}^{p+1}}.
\end{equation}
By \eqref{e58} and \eqref{e59} we have that
\begin{equation}\label{e60}
\int_{\Omega}v_{\Omega}^p\ge \frac{4p}{(p+1)^2}\lambda_1(\Omega)\int_{\Omega}v_{\Omega}^{p+1}.
\end{equation}
Multiplying both sides of \eqref{e60} by $\lambda^p_1(\Omega)$ gives that
\begin{equation}\label{e61}
\frac{4p}{(p+1)^2}F_{p+1}^{p+1}(\Omega)\le F_p^p(\Omega).
\end{equation}
Taking the supremum over all open $\Omega\subset \R^m$ with measure $1$, in the right-hand side of \eqref{e61}, and subsequently in the left-hand side of \eqref{e61} gives that
\begin{equation}\label{e62}
\mathfrak{F}_{p+1}^{p+1}\le \frac{(p+1)^2}{4p}\mathfrak{F}_p^p.
\end{equation}
Iterating \eqref{e62} $n-1$ times we find \eqref{e21}.
The same calculation carries over when $\Omega$ is an open, bounded convex set. This proves \eqref{e21a}.
By part (ii) we have that for $1\le p\le p_m,\, \mathfrak{F}_p=1$. This, together with \eqref{e21}, gives \eqref{e21b}.

\noindent(vi) Since $\mathfrak{F}_1=1$, we consider the case $n\in N,\,n\ge 2$. Put $p=1$ in \eqref{e21}, and replace $n$ by $n-1$. This gives \eqref{e21c}.

\noindent(vii) Substituting $p=1+\delta,\,0<\delta\le 1$ in \eqref{e21b} gives that
\begin{equation*}
\delta^{-1}\big(\mathfrak{F}_{2+\delta}-\mathfrak{F}_2\big)\le \delta^{-1}\bigg(\bigg(1+\frac{\delta^2}{4}\bigg)^{\frac{1}{2+\delta}}-1\bigg),\,
\end{equation*}
and the assertion follows by L' H\^{o}pital's rule.

\noindent(viii)
Taking suprema in \eqref{e12} over all bounded convex open sets $\Omega$ yields \eqref{e21d}.

\noindent(ix) Let $m = 2$. By \eqref{e21d}, and the numerical estimate (1.10) in \cite{MvdB} we have for $p=1+\delta,0<\delta\le 1$, that
\begin{equation}\label{e63}
\mathfrak{F}_{1+\delta}^{\textup{convex}}\le \bigg(1-\frac{1}{11560}\bigg)^{\frac{1}{1+\delta}}.
\end{equation}
By \eqref{e21a} for $n=1$ we have that
\begin{align}\label{e65}
\mathfrak{F}_{2+\delta}^{\textup{convex}}&\le \bigg(\frac{(2+\delta)^2}{4+4\delta}\bigg)^{\frac{1}{2+\delta}}\bigg(\mathfrak{F}_{1+\delta}^{\textup{convex}}\bigg)^{\frac{1+\delta}{2+\delta}}\nonumber \\ &
\le \bigg(1+\frac{\delta^2}{4}\bigg)^{\frac{1}{2+\delta}}\bigg(\mathfrak{F}_{1+\delta}^{\textup{convex}}\bigg)^{\frac{1+\delta}{2+\delta}}\nonumber \\ &
\le \bigg(1+\frac{\delta^2}{4}\bigg)^{\frac{1}{2+\delta}}\bigg(1-\frac{1}{11560}\bigg)^{\frac{1}{2+\delta}},
\end{align}
where we have used \eqref{e63} in the last inequality. Since the right-hand side of \eqref{e65} is equal to $1$ for $\delta^*=\frac{2}{(11559)^{1/2}}$,
we conclude that
\begin{equation*}
p_2^{\textup{convex}}\ge 2+\delta^*,
\end{equation*}
which proves the assertion in (ix).
\hspace*{\fill }$\square $

\section{A two-parameter family of inequalities \label{sec4}}

As mentioned at the end of the Introduction one can define a two-parameter family of products involving $T_p(\Omega),$ $\lambda_1(\Omega),$ and
$|\Omega|$, which is invariant under homotheties.

\begin{definition}\label{def3}
For an open set $\Omega\subset \R^m$ with finite Lebesgue measure, $p\ge 1,\, q\in \R$,
\begin{itemize}
\item[\textup{(i)}]
\begin{equation}\label{e74}
F_{p,q}(\Omega)=\frac{T_p(\Omega)\lambda^q_1(\Omega)}{|\Omega|^{\frac1p+\frac2m(1-q)}},
\end{equation}
\item[\textup{(ii)}]
\begin{equation}\label{e75}
F_{\infty,q}(\Omega)=\frac{\Vert v_{\Omega}\Vert_{L^{\infty}(\Omega)}\lambda^q_1(\Omega)}{|\Omega|^{\frac2m(1-q)}},
\end{equation}
\item[\textup{(iii)}]
\begin{equation}\label{e76}
\mathfrak{F}_{p,q}=\sup\{F_{p,q}(\Omega):\Omega\,\, \textup{open in}\, \R^m, \,
|\Omega|<\infty\},
\end{equation}
\item[\textup{(iv)}]
\begin{equation}\label{e77}
\mathfrak{F}_{\infty,q}=\sup\{F_{\infty,q}(\Omega):\Omega\,\, \textup{open in}\, \R^m, \,
|\Omega|<\infty\}.
\end{equation}
\end{itemize}
\end{definition}

It is straightforward to verify that the quantities defined in \eqref{e74} and \eqref{e75} are invariant under homotheties of $\Omega$. Below we characterize
those pairs $\{(p,q):p\ge 1\}$ for which the sharp constants defined in \eqref{e76} and \eqref{e77} are finite.
\begin{theorem}\label{the5}
\begin{itemize}
\item[\textup{(i)}]For $1\le p<\infty$, $\mathfrak{F}_{p,q}<\infty$ if and only if $q\le 1$.

\item[\textup{(ii)}]For $p=\infty$, $\mathfrak{F}_{\infty,q}<\infty$ if and only if $q\le 1$.
\end{itemize}
\end{theorem}
\begin{proof}
\noindent(i) We first suppose $q>1,\, 1\le p<\infty$. Let $\Omega_n$ be the disjoint union of $n$ balls with equal radii $r_n$, where $|\Omega_n|=\omega_mnr_n^m=1$.
Then $\lambda_1(\Omega_n)=r_n^{-2}\lambda_1(B_1)$. By scaling we have that
\begin{equation}\label{e78}
T_p^p(\Omega_n)=r_n^{2p}|B_1|^{-1}T_p^p(B_1).
\end{equation}
Hence by \eqref{e78},
\begin{equation}\label{e79}
\mathfrak{F}^p_{p,q}\ge F_{p,q}^p(\Omega_n)=|B_1|^{-1}r_n^{2p-2pq}T_p^p(B_1)\lambda_1^{pq}(B_1).
\end{equation}
Since $q>1$ and $r_n\downarrow 0$ as $n\rightarrow \infty$, we have that the right-hand side of \eqref{e79} tends to infinity as $n\rightarrow \infty$.

Next suppose $q\le 1,\, 1\le p<\infty$. By \eqref{e74}, Faber-Krahn, and Theorem \ref{the1}
\begin{align*}
F_{p,q}(\Omega)&=\frac{T_p(\Omega)\lambda_1(\Omega)}{|\Omega|^{\frac1p}}\lambda^{q-1}_1(\Omega) |\Omega|^{\frac2m(q-1)}\nonumber \\ &\le
\mathfrak{F}_p\lambda^{q-1}_1(B_1) |B_1|^{\frac2m(q-1)}\nonumber \\ &\le \mathfrak{F}_{\infty}\lambda^{q-1}_1(B_1) |B_1|^{\frac2m(q-1)}.
\end{align*}
This proves part (i).

\noindent(ii) We first suppose $q>1$, and let $\Omega_n$ be the set as in the proof of part (i) above. Then $\Vert v_{\Omega_n}\Vert_{L^{\infty}(\Omega_n)}=\frac{r_n^2}{2m}$.
Hence
\begin{align*}
\mathfrak{F}_{\infty,q}&\ge\frac{r_n^{2-2q}\lambda_1^q(B_1)}{2m},
\end{align*}
which tends to infinity as $r_n$ tends to $0$.

Next suppose $q\le 1.$ By \eqref{e75}, \eqref{e4} and Faber-Krahn,
\begin{align*}
F_{\infty,q}(\Omega)&=\frac{\Vert v_{\Omega}\Vert_{L^{\infty}(\Omega)}\lambda^q_1(\Omega)}{|\Omega|^{\frac2m(1-q)}}\nonumber \\ &\le
\mathfrak{F}_{\infty}\lambda^{q-1}_1(\Omega)|\Omega|^{\frac2m(q-1)}\nonumber \\ &\le \mathfrak{F}_{\infty}\lambda^{q-1}_1(B_1)|B_1|^{\frac2m(q-1)}.
\end{align*}
This proves part (ii).
\end{proof}

In general it looks very difficult to compute $\mathfrak{F}_{p,q}$ or even $\mathfrak{F}_p=\mathfrak{F}_{p,1}, p>2$, with the exception of $\mathfrak{F}_{p,0}$.
G. Talenti in \cite {Talenti} obtained a pointwise estimate between the rearrangement of the torsion function of a generic set with finite measure and the torsion function of the ball with the same measure. In particular this estimate implies that the $L^p$ norm of the torsion function is maximised by the $L^p$ norm of the torsion function for the ball with the same measure. Hence, by \eqref{e74} and \eqref{e75} we have
\begin{align*}
\mathfrak{F}_{p,0}=\frac{T_p(B_1)}{|B_1|^{\frac1p+\frac2m}}.
\end{align*}

However, in the one-dimensional case we have the following result.
\begin{theorem}\label{the5} If $m=1,\, q\le 1,\, 1\le p<\infty$, then
\begin{equation}\label{e84}
\mathfrak{F}_{p,q}=\frac{\pi^{(4pq+1)/(2p)}}{2^{(1+3p)/p}}\bigg(\frac{\Gamma(p+1)}{\Gamma(p+\frac32)}\bigg)^{1/p},
\end{equation}
and
\begin{equation}\label{e85}
\mathfrak{F}_{\infty,q}=\frac{\pi^{2q}}{8}.
\end{equation}
\end{theorem}

\noindent{\it Proof of Theorem \ref{the5}.}
Since $\Omega\subset\R^1$ is open it is a countable union of open intervals. Since $|\Omega|<\infty$, we let $2a_1\ge 2a_2\ge ...$ be the lengths of these intervals.
Without loss of generality we may assume that $|\Omega|=2\sum_{j=1}^{\infty}a_j=1$. By the first equality in \eqref{e44} we have by scaling for a single interval $B_a$ of length $2a$ that
\begin{align}\label{e66}
T_p(B_a)&=\frac{a^{(2p+1)/p}}{2}\bigg(2\int_{[0,1]}dr\, (1-r^2)^p\bigg)^{1/p}\nonumber \\ &
=\frac{a^{(2p+1)/p}\pi^{1/(2p)}}{2}\bigg(\frac{\Gamma(p+1)}{\Gamma(p+\frac32)}\bigg)^{1/p}\nonumber \\ &=a^{(2p+1)/p}c_p,
\end{align}
where $c_p$ can be read-off from \eqref{e66}.
Since $T_p^p$ is additive on disjoint open sets we have that
\begin{align*}
T^p_p(\Omega)&=c^p_p\sum_{j=1}^{\infty}a_j^{2p+1}\le c^p_pa_1^{2p}\sum_{j=1}^{\infty}a_j=2^{-1}c^p_pa_1^{2p}.
\end{align*}
Since
\begin{equation*}
\lambda_1(\Omega)=\frac{\pi^2}{4a_1^2},
\end{equation*}
$q\le 1$, and $2a_1\le 1$, we have that $(2a_1)^{2-2q}\le 1$. Hence
\begin{equation*}
F_{p,q}(\Omega)\le 2^{-1/p}c_p\bigg(\frac{\pi^2}{4}\bigg)^qa_1^{2-2q}\le 2^{-(1+2p)/p}\pi^{2q}c_p.
\end{equation*}
By taking the supremum over all $\Omega\subset \R^1$ with measure $1$ we obtain that
\begin{equation}\label{e70}
\mathfrak{F}_{p,q}\le2^{-(1+2p)/p}c_p\pi^{2q}.
\end{equation}
To obtain a lower bound for $\mathfrak{F}_{p,q}$ we make the particular choice of $\Omega=B_1$. This gives that
\begin{equation}\label{e71}
\mathfrak{F}_{p,q}\ge F_{p,q}(B_1)=2^{-(1+2p)/p}\pi^{2q}c_p.
\end{equation}
By \eqref{e70} and \eqref{e71}
we conclude that
\begin{equation}\label{e72}
\mathfrak{F}_{p,q}= F_{p,q}(B_1)=2^{-(1+2p)/p}\pi^{2q}c_p.
\end{equation}
and \eqref{e84} follows from \eqref{e72} and the definition of $c_p$ in \eqref{e66}.

To prove \eqref{e85} we just observe that the maximum of the torsion function and the first Dirichlet eigenvalue are determined by the largest interval in $\Omega$, i.e. $a_1$. Since $q\le 1$ we maximise the resulting expression by taking $a_1=\frac12$.
\hspace*{\fill }$\square $

Note that as $B_1$ is convex we also have that
\begin{equation}\label{e73}
\mathfrak{F}_p^{\textup{convex}}=\mathfrak{F}_{p,1}=\mathfrak{F}_p,
\end{equation}
and recover the known values $\mathfrak{F}_1=\frac{\pi^2}{12}, \mathfrak{F}_{\infty}=\frac{\pi^2}{8}$, (\cite{MvdB}, \cite{MvdB1}). Note that $\mathfrak{F}_1<\mathfrak{F}_2=\frac{\pi^2}{\sqrt{120}}<1$, which is in contrast with the higher dimensional situation $m\ge 2$, where $\mathfrak{F}_p=1,\,1\le p\le 2$.

\end{document}